\documentclass[a4paper]{article}
\usepackage{latexsym}
\usepackage{amsmath,amssymb,amsfonts}
\renewcommand{\baselinestretch}{1.2}
\newtheorem{prethm}{{\bf Theorem}}

\newenvironment{thm}{\begin{prethm}{\hspace{-0.5
               em}{\bf.}}}{\end{prethm}}

\newtheorem{prepro}[prethm]{Proposition}

\newtheorem{prelem}[prethm]{Lemma}

\newtheorem{precor}[prethm]{Corollary}

\newtheorem{preremark}{{\bf Remark}}

\newenvironment{rem}{\begin{preremark}\em{\hspace{-0.5
              em}{\bf.}}}{\end{preremark}}

\newtheorem{preexample}{{\bf Example}}

\newtheorem{preproblem}{{\bf problem}}

\newtheorem{preproof}{{\bf Proof.}}

\newenvironment{proof}[1]{\begin{preproof}{\rm
               #1}\hfill{$\Box$}}{\end{preproof}}

\renewcommand{\thefootnote}

\begin{document}

\date{}
\title{A note on comaximal graph of non-commutative rings}
\author{S. Akbari$^{\,\rm a}$,~ M. Habibi$^{\,\rm b}$,~ A. Majidinya$^{\,\rm b}$,~ R. Manaviyat$^{\,\rm b}$,
\\
{\footnotesize {\em $^{\rm a}$Department of Mathematical Sciences, Sharif University of Technology,}}\\
{\footnotesize {\em Tehran, Iran}}\\
{\footnotesize{\em $^{\rm b}$Department of Pure Mathematics, Faculty of Mathematical Sciences,}}\\
{\footnotesize {\em Tarbiat Modares University, Tehran, Iran}}}
\footnotetext{E-mail Addresses: {\tt s\_akbari@sharif.edu}, {\tt
mhabibi@modares.ac.ir}, {\tt a.majidinya@modares.ac.ir}, {\tt
r.manaviyat@modares.ac.ir}}
\date{}
\maketitle
\begin{quote}
{\small \hfill{\rule{13.3cm}{.1mm}\hskip2cm} \textbf{Abstract}\\
 Let $R$ be a ring with unity. The graph $\Gamma(R)$ is a graph with vertices as elements of $R$, where two distinct
 vertices $a$ and $b$ are adjacent if and only if $Ra+Rb=R$. Let $\Gamma_2(R)$ is the subgraph of $\Gamma(R)$ induced
 by the non-unit elements. H.R. Maimani et al. [H.R. Maimani et al., Comaximal graph of commutative rings, J. Algebra $319$
 $(2008)$ $1801$-$1808$] proved that: ``If $R$ is a commutative ring with unity and the graph
 $\Gamma_2(R)\backslash J(R)$ is $n$-partite, then the number of maximal ideals of $R$ is at most $n$."
 The proof of this result is not correct. In this paper we present
a correct proof for this result. Also we generalize some results
given in the aforementioned paper for the non-commutative rings.
\vspace{1mm} {\renewcommand{\baselinestretch}{1}
\parskip = 0 mm

}}

\vspace{-3mm}\hfill{\rule{13.3cm}{.1mm}\hskip2cm}
\end{quote}

\section{Introduction}
{Throughout this paper $R$ denotes an associative ring with unity and $J(R)$
denotes the Jacobson radical of $R$. We also denote $M_n(R)$ and $\mathrm{Max}(R)$, for
the set of $n\times n$ matrices over the ring $R$ and the set of all maximal left ideals of $R$, respectively.
In this paper $\mathbb{F}_q$ denotes the field with $q$ elements. An $n$-$partite$ graph is one whose vertex set
can be partitioned into $n$ subsets so that no edge has both ends in any one subset. A $complete$ $n$-$partite$
graph is one in which each vertex is joined to every vertex that is not in the same subset. A $clique$ of a graph
$G$, is a complete subgraph of $G$.
In $\cite{sha}$, Sharma and Bhatwadekar defined  $comaximal$ $graph$, $\Gamma(R)$, with vertices as elements of
$R$, where two distinct vertices $a$ and $b$ are adjacent if and only if
$Ra+Rb=R$. The Comaximal graph of rings has been studied by several authors, see $\cite{yas,sha,wan1,wan2}$.
The subgraph of $\Gamma(R)$ induced on unit elements of $R$ and the subgraph of $\Gamma(R)$ induced on non-unit
elements of $R$ denoted by $\Gamma_1(R)$ and $\Gamma_2(R)$, respectively. In $\cite{yas}$, some properties of
$\Gamma_2(R)\backslash J(R)$, for a commutative ring $R$ have been studied. The following theorem was proved in
$\cite{yas}$.
\begin{thm}
 Let $R$ be a commutative ring with unity and $n>1$. Then the following hold:
\begin{enumerate}
\item [\rm{(a)}] If $|\mathrm{Max}(R)|=n<\infty$, then the graph $\Gamma_2(R)\backslash J(R)$ is $n$-partite.

\item [\rm{(b)}] If the graph  $\Gamma_2(R)\backslash J(R)$ is $n$-partite, then $|\mathrm{Max}(R)|\leq n$.
In this case if the graph $\Gamma_2(R)\backslash J(R)$ is not $(n-1)$-partite, then $\mathrm{Max}(R)=n$.\\
\end{enumerate}
\end{thm}
In the proof of the previous theorem, Part $(b)$, the authors by contradiction assume that $|\mathrm{Max}(R)|>n$
and then they claim that if $m_1,\ldots,m_{n+1}\in \mathrm{Max}(R)$
and $x_i\in m_i\backslash\bigcup\limits_{j\ne i}{m_j}$, then $\{x_1,\ldots,x_{n+1}\}$ is a clique.
This claim is not true because consider the ring $R=\prod\limits_{i=1}^\infty{\mathbb{Z}_2}$. Let $m_1$ and $m_2$
be the set of all elements of $R$ whose the first and second entries are zero, respectively. Then
$e_1=(1,0,0,\ldots)\in m_2\backslash m_1$ and $e_2=(0,1,0,\ldots)\in m_1\backslash m_2$, but $e_1$ and $e_2$
are not adjacent.

In the following, we provide a correct proof for this part.

\begin{thm}
Let $R$ be a commutative ring with unity and $|\mathrm{Max}(R)|\geq 2$. If $|\mathrm{Max}(R)|\geq n$, then
$\Gamma_2(R)\backslash J(R)$ has a clique of order $n$.
\end{thm}


\begin{proof}
{Let $\{m_1,\ldots,m_n\}\subseteq Max(R)$. First we claim that for every
$x_1\in m_1\backslash \bigcup\limits_{j=2}^{n}{m_j}$, there exists a clique in
$\Gamma_2(R)\backslash J(R)$ with the vertex set $\{x_1,\ldots,x_n\}$, where $x_i\in m_i$, for $i=1,\ldots,n$.
We apply induction on $n$. Clearly, for $n=2$ the assertion is true.
By Prime Avoidance Theorem $[1,$ $p.8]$, $m_1\cap m_n\nsubseteq\bigcup\limits_{j=2}^{n-1}{m_j}$. So
there exists $y\in (m_1\cap m_n)\backslash\bigcup\limits_{j=2}^{n-1}{m_j}$.
Thus $x_1y\in (m_1\cap m_n)\backslash\bigcup\limits_{j=2}^{n-1}{m_j}$. By induction hypothesis there exists a
clique with vertex set $\{x_1y,x_2,\ldots,x_{n-1}\}$,
where $x_i\in m_i\backslash \bigcup\limits_{\scriptstyle j=1\hfill \atop\scriptstyle j\neq i\hfill}^{n-1}{m_j}$,
$2\leq i\leq n-1$. Since $x_1y\in m_n$, thus $x_2,\ldots,x_{n-1}\notin m_n$. Clearly, $x_1$ is adjacent to
$x_2,\ldots,x_{n-1}$. On the other hand $x_1x_2\ldots x_{n-1}\notin m_n$. This implies that there exists $x_n\in m_n$
which is adjacent to $x_1x_2\ldots x_{n-1}$. Thus $\{x_1,\ldots,x_n\}$ is a clique of order $n$ in
$\Gamma_2(R)\backslash J(R)$ and the proof is complete.}
\end{proof}


The next theorem was proved in $\cite{yas}$.

\begin{thm}
Let $R$ be a commutative ring with unity and $|\mathrm{Max}(R)|\geq 2$. Then the following hold:
\begin{enumerate}
\item [\rm{(a)}] If $\Gamma_2(R)\backslash J(R)$ is a complete $n$-partite graph, then $n=2$.

\item [\rm{(b)}] If there exists a vertex of $\Gamma_2(R)\backslash J(R)$ which is adjacent to every vertex,
then $\cong \mathbb{Z}_2\times F$, where $F$ is a field.
\end{enumerate}
\end{thm}

In this paper we will generalize the previous theorem for the
non-commutative rings. Before stating our results, we need the
following remark.


\begin{rem}
Let $R$ be a ring and $|\mathrm{Max}(R)|\geq 2$. Recall that an
element $x\in R$ is left-invertible in $R$ if and only if
$\overline{x}$ is left-invertible in $\overline{R}=\frac{R}{J(R)}$,
see $[2,$ $p.52]$. On the other hand $Rx+Ry=R$ if and only if
$\overline{R}\overline{x}+\overline{R}\overline{y}=\overline{R}$. So
$\Gamma_2(R)\backslash J(R)\cong \Gamma_2(\overline{R})\backslash
\{0\}$.
\end{rem}


\begin{thm}
Let $R$ be a ring and $|\mathrm{Max}(R)|\geq 2$. If
$\Gamma_2(R)\backslash J(R)$ is a complete $n$-partite graph, then $n=2$ or $n=q+1$, where $q$ is a
power of a prime number. Moreover, $\frac{R}{J(R)}\cong M_2(\mathbb{F}_q)$ or $\frac{R}{J(R)}\cong D_1\times D_2$,
where $D_1$ and
 $D_2$ are division rings.
\end{thm}
\begin{proof}{
Let $m_i$ and $m_j$ be two distinct maximal left ideals of $R$.
First note that no two distinct elements of $m_i\backslash J(R)$ are
adjacent. Also every element of $m_i\backslash m_j$ is adjacent to
at least one element of $m_j\backslash m_i$. Since
$\Gamma_2(R)\backslash J(R)$ is a complete $n$-partite graph,
$m_i\backslash m_j$ and $m_j\backslash m_i$ are two subsets of
distinct parts of $\Gamma_2(R)\backslash J(R)$. No element of
$(m_i\backslash J(R))\cap (m_j\backslash J(R))$ is adjacent to
$m_i\backslash m_j$ or $m_j\backslash m_i$. This implies that
$(m_i\backslash J(R))\cap (m_j\backslash J(R))=\varnothing$.
Therefore $m_i\cap m_j=J(R)$. Since $\Gamma_2(R)\backslash J(R)$ is
a complete $n$-partite graph, $R$ has exactly $n$ maximal left
ideals. Suppose that $\mathrm{Max}(R)=\{m_1,\ldots,m_n\}$. Consider
the natural left $R$-module monomorphism
$f:\frac{R}{J(R)}\rightarrow\frac{R}{m_1}\times\cdots\times\frac{R}{m_n}$.
Since $\frac{R}{m_i}$ is an Artinian left $R$-module, ${\rm
im}(f)\cong\frac{R}{J(R)}$ is a left Artinian ring. So by
Wedderburn-Artin Theorem, $[2,$ $p.33]$, $\frac{R}{J(R)}\cong
M_{n_1}(D_1)\times\cdots\times M_{n_k}(D_k)$, where $ D_1,\ldots
,D_k$ are division rings. Every distinct maximal left ideals of
$\frac{R}{J(R)}$ intersect each other trivially. So $k\leq 2$ and
$\frac{R}{J(R)}\cong M_n(D)$ or $\frac{R}{J(R)}\cong D_1\times D_2$.
Suppose that $\frac{R}{J(R)}\cong M_n(D)$ and $n\geq 3$. Let $m_i$
be the maximal left ideal of $M_n(D)$ containing all matrices  whose
$i^{th}$ column is zero, for $i=1,2$. Clearly, $m_1\cap m_2\neq 0$,
a contradiction. So $\frac{R}{J(R)}\cong D_1\times D_2$ or
$\frac{R}{J(R)}\cong M_2(D)$. Let $\frac{R}{J(R)}\cong M_2(D)$ and
$M=\left\{\left[
\begin{array}{cc}
 a & 0 \\
 b & 0 \\
\end{array}
\right]
\mid\,\ a,b\in D\,\right\}$. Clearly, $M$ is a maximal left ideal of $M_2(D)$. For every $\alpha \in D$, let
$$M_{\alpha}=
\left[
\begin{array}{cc}
 1 & \alpha \\
 0 & 1 \\
\end{array}
\right]M\left[
\begin{array}{cc}
 1 & \alpha \\
 0 & 1 \\
\end{array}
\right]^{-1}=\left\{\left[
\begin{array}{cc}
 a+\alpha b & -a\alpha-\alpha b \alpha \\
 b & -b \alpha \\
\end{array}
\right]\mid\,\ a,b\in D\,\right\}$$

Since $\frac{-b\alpha}{b}=-\alpha$, we conclude that for every distinct elements $\alpha ,\alpha'\in D$,
$M_{\alpha}\neq M_{\alpha'}$. Since $M_2(D)$ has finitely many maximal left ideals, so by Wedderburn's ``Little"
Theorem, $[2,$ $p.203]$, $D\cong\mathbb{F}_q$. Now, by $[6,$ Lemma $4.2]$, the number of maximal left ideals of
$M_2(\mathbb{F}_q)$ is $q+1$ and the proof is complete.  }
\end{proof}


\begin{rem}
Let $R=M_2(\mathbb{F}_q)$. Then we show that $\Gamma_2(M_2(\mathbb{F}_q))\backslash \{0\}$ is a complete
$(q+1)$-partite graph. To see this assume that $m_1$ and $m_2$ be two distinct maximal left ideals of $R$.
By $[6,$ Lemma $4.2]$, for every $0\neq a\in m_1$ and $0\neq b\in m_2$, $Ra=m_1$ and $Rb=m_2$.
Thus $\Gamma_2(M_2(\mathbb{F}_q))\backslash \{0\}$ is a complete $(q+1)$-partite graph. Note that $M_{\alpha}=
\left\{\left[
\begin{array}{cc}
 a+\alpha b & -a\alpha-\alpha b \alpha \\
 b & -b \alpha \\
\end{array}
\right]\mid\,\ a,b\in \mathbb{F}_q\,\right\}$, $\alpha\in \mathbb{F}_q$ and $M'=\left\{\left[
\begin{array}{cc}
 0 & a \\
 0 & b \\
 \end{array}
 \right]\mid\,\ a,b\in \mathbb{F}_q\,\right\}
$ are $q+1$ distinct maximal left ideals of $M_2(\mathbb{F}_q)$.
\end{rem}


\begin{thm}
Let $R$ be a ring and $|\mathrm{Max}(R)| \geq 2$. If there exists a
vertex of $\Gamma_2(R)\backslash J(R)$ adjacent to every
other vertices, then $R \cong \mathbb{Z}_2\times D$, where $D$ is a
division ring.
\end{thm}


\begin{proof}{
 Let $x$ be a vertex of $\Gamma_2(R)\backslash J(R)$ which
is adjacent to all other vertices of $\Gamma_2(R)\backslash J(R)$.
First we claim that $J(R)=0$. If $0\neq a \in J(R)$, then clearly
$x+a$ is a vertex of $\Gamma_2(R)\backslash J(R)$ and
$Rx+R(x+a)\subseteq  Rx+Ra\subseteq m$, for some maximal left ideal
$m$ of $R$. So $x$ and $x+a$ are not adjacent, a contradiction. Thus
$J(R)=0$. Now, we show that $Rx=\{0,x\}$. If $Rx\neq \{0,x\}$, then
there exists an element $r \in R$ such that $0\neq rx\neq x$.
Clearly, $rx\in\Gamma_2(R)\backslash \{0\}$ . But $x$ and $rx$ are
not adjacent, a contradiction. Thus $Rx=\{0,x\}$. For every non-unit
element $y\notin Rx$, since $y$ is adjacent to $x$, $Rx$ is a
maximal left ideal of $R$. If $x^2=0$, then $Rx$ is a nilpotent left
ideal of $R$ and by $[2,$ $p. 53]$, we have $Rx\subseteq
J(R)=\{0\}$, a contradiction. So $x^2=x$ is an idempotent. Now, we
have $R=Rx\oplus R(1-x)$ and so $R(1-x)$ is a simple left
$R$-module. Hence $R$ is a semisimple ring and by Wedderburn-Artin
Theorem, $R\cong M_{n_1}(D_1)\times \cdots\times M_{n_k}(D_k)$,
where $D_1,\ldots ,D_k$ are division rings. Since $R$ is direct sum
of two simple left $R$-modules, $k=2$. Also since $R$ has a maximal
left ideal $m$ with two elements, $R$ has $\mathbb{Z}_2$ as a simple
component. If $n_k\neq1$, then $m$ properly contained in a left
ideal of $R$, a contradiction. Thus $n_k=1$. This implies that
$R\cong \mathbb{Z}_2\times D$, where $D$ is a division ring and the
proof is complete. }
\end{proof}

\noindent {\bf Acknowledgement.} The first author is indebted to the Research Council of Sharif University of Technology for
support.


\end{document}